\documentclass[leqno,12pt]{article}

\usepackage{graphicx}
\usepackage{amssymb}
\usepackage{pstricks}
\usepackage{amsmath}
\usepackage{amsfonts}
\usepackage{fancybox}
\usepackage{amsthm}
\usepackage{verbatim}
\usepackage{mathrsfs} 
\usepackage{amssymb}
\usepackage{euscript}

\usepackage[left=1in,top=1in,right=1in,bottom=1in]{geometry}
{\begin{Sbox}\begin{minipage}}%
{\end{minipage}\end{Sbox}\fbox{\TheSbox}}

\title{Interfacial Phenomena and Natural Local Time}
\author{ Appuhamillage, Thilanka A\footnote{Department of Mathematics, Oregon State University.  
$^\dagger$ School of Chemical, Biological and Environmental Engineering, Oregon State University.
$^+$ Corresponding author:
waymire@math.oregonstate.edu.  }
\and Bokil, Vrushali A$^*$
\and Thomann, Enrique A$^*$
\and Waymire, Edward C$^{*+}$
\and Wood, Brian D$^\dagger$
}
\date{February 8, 2012}

\begin{document}

\maketitle

\baselineskip=18pt \setcounter{page}{1}

\renewcommand{\theequation}{\thesection.\arabic{equation}}
\newtheorem{theorem}{Theorem}[section]
\newtheorem{lemma}[theorem]{Lemma}
\newtheorem{proposition}[theorem]{Proposition}
\newtheorem{corollary}[theorem]{Corollary}
\newtheorem{remark}[theorem]{Remark}
\newtheorem{fact}[theorem]{Fact}
\newtheorem{problem}[theorem]{Problem}
\newtheorem{example}[theorem]{Example}
\newtheorem{question}[theorem]{Question}
\newtheorem{conjecture}[theorem]{Conjecture}
\newtheorem{definition}[theorem]{Definition}

\theoremstyle{Remark}

\newcommand{\eqnsection}{
\renewcommand{\theequation}{\thesection.\arabic{equation}}
    \makeatletter
    \csname  @addtoreset\endcsname{equation}{section}
    \makeatother}
\eqnsection
\def\sgn{{\text sgn}}
\def\r{{\mathbb R}}
\def\e{{\mathbb E}}
\def\p{{\mathbb P}}
\def\P{{\bf P}}
\def\E{{\bf E}}
\def\Q{{\bf Q}}
\def\z{{\mathbb Z}}
\def\N{{\mathbb N}}
\def\T{{\mathbb T}}
\def\G{{\mathbb G}}
\def\L{{\mathbb L}}

\def\deg{\chi}

\def\ee{\mathrm{e}}
\def\d{\, \mathrm{d}}
\def\S{\mathscr{S}}


\def\C{\text{Cov}}
\def\V{\text{Var}}



\begin{abstract}
This article addresses a modification of local time for stochastic processes, to be referred to as {\it natural local time}.  It is prompted by  theoretical developments 
 arising in mathematical treatments of recent experiments and observations 
 of phenomena in the geophysical and biological sciences pertaining to dispersion 
in the presence of an interface of discontinuity in dispersion coefficients.  The results illustrate new ways in which to use the theory of stochastic processes to infer macro scale parameters and behavior from micro scale observations in particular heterogeneous environments. 
\end{abstract}

\bigskip

\noindent{\slshape\bfseries Keywords.} Advection-dispersion, discontinuous diffusion, skew Brownian motion, mathematical local time, natural local time

\bigskip

\noindent{\slshape\bfseries 2010 Mathematics Subject Classification.} 


\bigskip
\bigskip

\section{Introduction and Motivating Example}
The purpose of this article is to present some new results that call attention to the presence and effects of particular types of heterogeneity observed to occur in diverse geophysical and ecological/biological spatial environments.   A few examples of interest to the authors are provided to highlight the occurrence  of interfacial discontinuities, 
but one may expect that readers will easily be able to conceive of many others.  Apart from the motivating examples,  theoretical results are provided to illustrate the 
interfacial effects in ways that may prove useful in analyzing and interpreting more complex data sets.  

The most basic empirical considerations involve various temporal measurements, 
e.g., breakthrough times, residence or occupation times, recurrence times and,
as will be explained, a fundamental notion of {\it local time}.

Mathematical local time is a quantity with a long history in the theory of stochastic processes.  
It concerns a mathematical measure of the amount
of time a stochastic process spends locally about a point, which we will refer to as {\it mathematical
local time} in distinction with a modification introduced here to be
 referred to as {\it natural local time}.\footnote{We reserve the use of the term ``physical local time''  for a more specialized case,  but physical 
 and/or biological modeling considerations
 underlie the 
development of the nomenclature introduced in this article.} 
Starting with the celebrated theorem of Trotter \cite{trotter}
establishing the continuity of mathematical
local time for standard Brownian motion, 
the general problem of determining necessary and sufficient conditions for continuity of 
local time for stochastic processes is an old and well-studied problem at the foundations
of probability theory; see  \cite{eisen_kaspi} for recent progress and insight into the 
technical depth of the problem for a  large class
of Markov processes.  

While the present paper has little to offer to the mathematical foundations of the subject,
observations from field and laboratory experiments are provided together with some theoretical results that point to new directions in this area from the perspective of applications and modeling of certain dispersive phenomena.
 At the most fundamental 
level, the issue involves the
{\it units} of measurement.   In particular,  the units of mathematical 
local time in the context of dispersion of particle concentrations are typically
 those of {\it spatial length}.  This is not unnatural in the mathematical
context owing to the locally linear relationship between spatial variance  and time when the dispersion coefficient is sufficiently smooth.  However, as will be seen, this breaks down in the presence of an interface of discontinuity in the diffusion coefficient.   

Effort is made to make this article accessible to a diverse audience by
including some basic mathematical definitions and background theory.  A complete and systematic treatment of essentially all of the underlying mathematical concepts can be found in  \cite{revuz_yor}.  We conclude this introductory section with an empirical example  that serves to
motivate and illustrate the role of a key mathematical construct,  namely that of {\it skew Brownian motion}.    This is followed by a section with additional empirical illustrations from the ecological and biological sciences.  The emphasis of the paper is on mathematical theory and modeling.  The main results demonstrate (i) the effects of  general point interface transmission probabilities on breakthrough times and (ii) on residence times; (iii) symmetry (via martingale) relationships between interfacial transmission probabilities, dispersion coefficients, and skewness parameters;
(iv) a special role for continuity of flux (conservation of mass) verses continuity of derivatives at an interface;
(v) a special role for (natural) local time continuity in the
determination of the transmission parameter at the interface.

\vskip .05in
\noindent{\bf Example A} ({\it\bf Dispersion in Porous Media}). 
The topic addressed in this paper  originally initiated as a result of 
questions resulting from recent
laboratory experiments designed to empirically test
and understand advection-dispersion in the presence of sharp interfaces. e.g., 
experiments by \cite{kuo_1999},  
\cite{hoteit}, 
\cite{Berkowitz09}.
Such laboratory experiments have 
been rather sophisticated in the use of layers of sands and/or
glass beads of different granularities  and modern measurement technology
 (see Figure \ref{brianslab}).

\begin{figure}[htb]
\label{brianslab}
\begin{center}
    \includegraphics[scale=0.7]{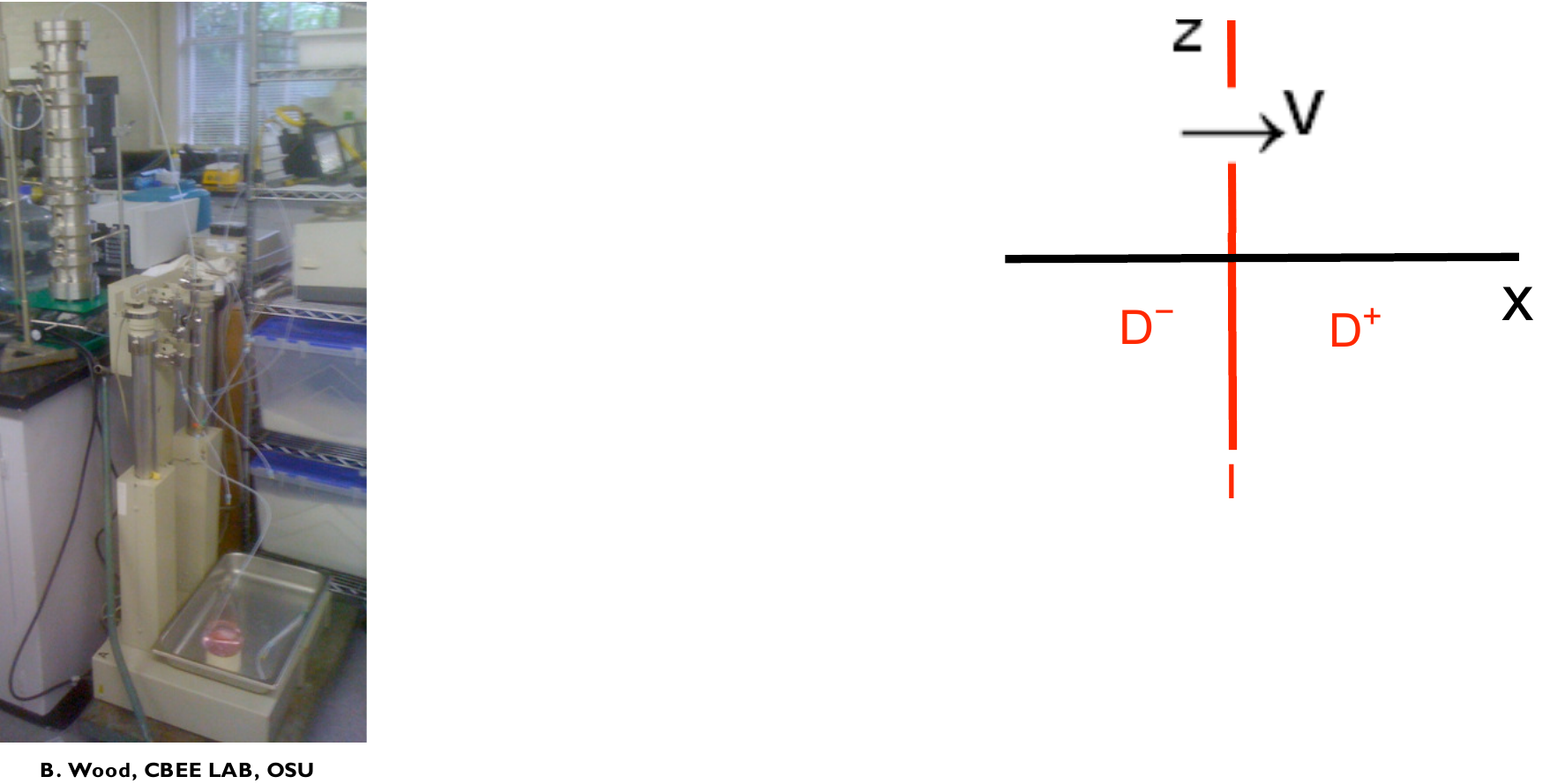}
  \end{center}
\caption{ Saturated Fine/Coarse Grain Sand Separated by an Interface Experimental Assembly, Oregon State University}
\end{figure}

From an engineering science point of view, basic research in this area is aimed 
to improve understanding of the role of sharp discontinuities
in the hydrologic parameters appearing in equations used for predictions of the spread of contaminants
 in saturated porous media. There is a huge literature in the sciences and
 engineering pertaining to the general problem.  For a perspective on the
 mathematical foundations see  \cite{bhatt},
 \cite{bhatt_gupta},\cite{bhatt_gotze}, 
  \cite{newman_winter}, \cite{papa_varad}, \cite{benarous}. The former
 considers heterogeneity from a deterministic framework, while the
 latter involve randomness in the medium.  \cite{cushman} contains
 a survey of various approaches, together with a more exhaustive set
 of references to the geophysical and hydrologic literature.  


From a general 
mathematical point of view an {\it interface} is defined by a hypersurface across which the dispersion coefficient is discontinuous. 
As is well-known for the case  of dilute suspensions
in a homogeneous medium (e.g., water),  perhaps flowing at a rate $v$,  the particle motion is that of  a  Brownian motion with a constant diffusion
coefficient $D > 0$ and drift $v$.  When applied to colloidal suspensions, this is the famous classic theoretical result of Albert Einstein \cite{einstein} 
experimentally confirmed by Jean Perrin \cite{perrin}; see \cite{newburgh}
for a more contemporary experiment.  The essential physics of the problem, however, is relatively unchanged when one considers dissolved chemical species. 

The basic phenomena of interest to us here is captured by the following:

\noindent {\bf Question.}  Suppose that a dilute solute is injected at a point $L$ units to the left of 
an interface at the origin and retrieved at a point $L$ units to the right of the interface.  Let $D^-$ denote
the (constant) dispersion coefficient to the left of the origin and $D^+$ that to the right, with
say $D^- < D^+$ (see Figure \ref{inhomogeneous}).   Conversely,
suppose the solute is injected at a point $L$ units to the right of the interface and retrieved at a point $L$
units to the left.  {\it In which of these two symmetric arrangements will the immersed solute most rapidly breakthrough at the opposite end ? } 

\begin{figure}[htb]
\label{inhomogeneous}
\begin{center}
    \includegraphics[scale=0.5]{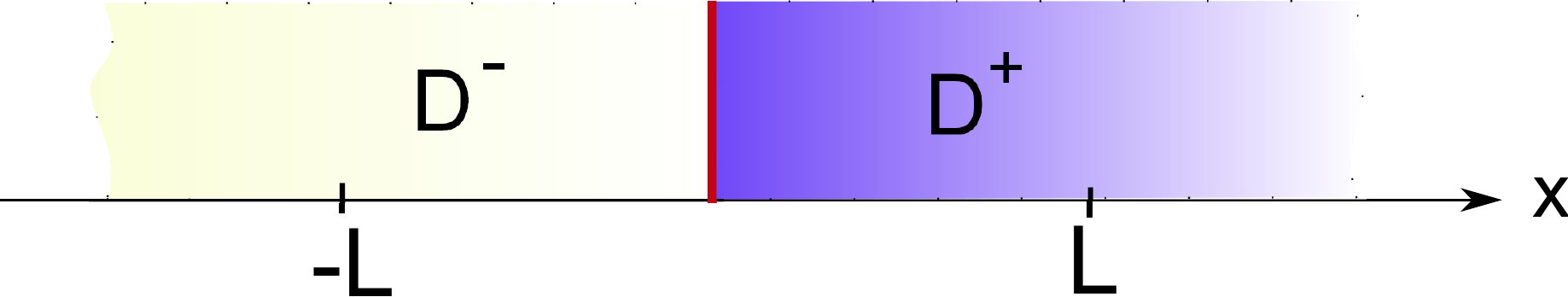}
  \end{center}
\caption{Interfacial Schematic}
\end{figure}
In the experimental set-up,  the coordinate
in the direction of flow is one-dimensional flow across an interface.  As will be explained, a localized point interface results in  a skewness effect that explains much of the empirically observed results suggested above; see 
\cite{Ramirez06}, \cite{Ramirez08}, \cite{App09a}, \cite{App09b}, \cite{Ramirez11a}, \cite{quastel}.

Let us briefly recall the notion of {\it skew Brownian motion} introduced by 
It\^o and McKean in
\cite{itomckean} to identify a class of stochastic processes corresponding to Feller's classification of
one-dimensional diffusions.  The simplest definition of skew Brownian motion $B^{(\alpha)}$
 for a given parameter, referred to as the 
 {\it transmission probability}, 
$\alpha\in[0,1]$ is as follows:  Let $|B|$ denote reflecting Brownian motion starting at zero,  and
enumerate the excursion intervals away from zero by $J_1,J_2,\dots$.
\begin{figure}[htb]
\label{skewbm}
\begin{center}
    \includegraphics[scale=0.4]{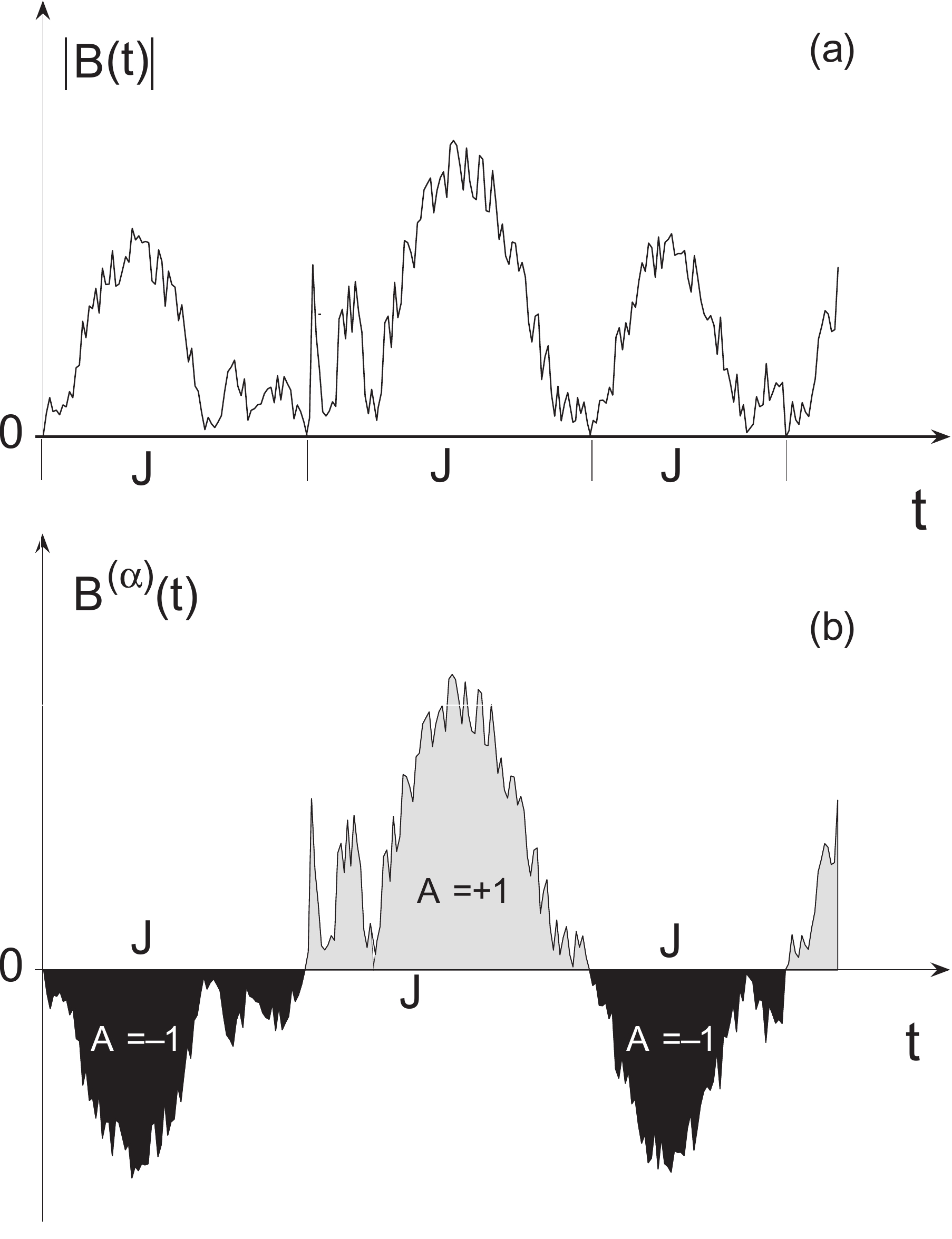}
  \end{center}
\caption{Skew Brownian Motion Construction}
\end{figure}

  This is possible because
Brownian motion has continuous paths and, therefore, its set of zeros is a closed set; the set's compliment is open and, 
hence, a countable disjoint union of open intervals.  Let $A_1,A_2,\dots$ be an
i.i.d. sequence of Bernoulli $\pm 1$ coin tossing random variables, independent of $B$, with 
$P(A_n=1) = \alpha$.  Then, $B^{(\alpha)}$ is obtained by changing the signs of the excursion over the
intervals $J_n$ whenever to $A_n=-1$, for $n = 1,2,\dots$ (see Figure \ref{skewbm}).
 That is,
\begin{equation}
B^{(\alpha)}(t) = \sum_{n=1}^\infty A_n1_{J_n}(t)|B(t)|, \qquad t\ge 0.
\end{equation}
In particular, the path $t\to B^{(\alpha)}(t), t\ge 0$ is continuous with probability one since the
Brownian paths are continuous.

For positive parameters $D^+, D^-$, consider a piecewise constant dispersion coefficient with interface
at $x = 0$ given by
$$D(x) = D^-1_{(-\infty,0)}(x) + D^+1_{[0,\infty)}(x). 
\quad x\in {\mathbb R},$$
Here $1_J$ denotes the indicator function of an interval $J$ defined by
$1_J(x) = 1, x\in J, 1_J(x) = 0, x\notin J$.
In response to the question about breakthrough  times raised above,  
the following theorem provides an answer as a consequence of 
 \cite{Ramirez06}, \cite{App09a}.  For convenience of mathematical
 notation take $v=0$ for here.
 
 \begin{remark}  Just as standard Brownian motion may be obtained as a limit of 
 rescaled simple symmetric random walks, e.g., \cite{bhattway1}, 
 an alternative description of skew Brownian motion with 
 transmission probability $\alpha$ may be obtained as a limit in distribution of 
 of the rescaled  {\it skew random walk} defined by $\pm 1$
 displacement probabilities,  $1/2$ each,  at  nonzero lattice points, and
 having  probabilities
 $\alpha$, $1-\alpha$ for transitions from  $0$ to $+1$ and $0$ to $-1$, respectively,
see  \cite{harrison_shepp}. 
 \end{remark}

\vskip .15in
\begin{theorem}
\label{thmasympassagetime}
 Let $D^+, D^-$ be arbitrary positive numbers, with  $D^- < D^+$.    
Define $Y^{(\alpha^*)}(t) = \sigma(B_t^{(\alpha^*)}), t\ge 0$, where $B^{(\alpha^*)}$ is skew Brownian motion with transmission
parameter $\alpha^* =  {\sqrt{D^+}\over \sqrt{D^+} + \sqrt{D^-}}$, and
 $\sigma(x) = \sqrt{D^+}x1_{[0,\infty)}(x) + \sqrt{D^-}x1_{(-\infty,0]}(x), x\in {\mathbb R}$.
Let $T_y = \inf\{t\ge 0:  Y^{(\alpha^*)}(t) = y\}.$  Then,

\vskip .05in

\noindent (a) For smooth  initial data  $c_0$,   \    $c(t,y) = E_yc_0(Y^{(\alpha)}(t) ), \quad t\ge 0$, solves 
$${\partial c\over\partial t} = {1\over 2}{\partial\over\partial y}(D(y){\partial c\over\partial y}),\quad D^+{\partial c(t,0^+)\over\partial y} = D^-{\partial c(t,0^-)\over\partial y}.$$

\vskip .05in

\noindent (b)  For $y > 0$,  $\quad P_{-y}(T_y > t) \le {\sqrt{D^-}\over\sqrt{D^+}} P_{y}(T_{-y} > t) <  P_{y}(T_{-y} > t) , \quad t\ge 0.$

\end{theorem}

\vskip .15in

 Observe, for example,  that by integrating the  complementary distribution functions in (b), one obtains that the mean breakthrough time in fine to coarse media is {\it smaller} than that of the mean breakthrough time a coarse to fine media.  
  Related phenomena and  results on dispersion 
in this context, including the case $v\neq 0$,
 are also given in \cite{Ramirez06}, \cite{Ramirez08}, \cite{appthesis}, 
\cite{App09a}, \cite{Ramirez11a}, \cite{quastel}.  
The proof of part (b) of the theorem relies on a transformation to 
 {\it elastic skew Brownian motion} to eliminate the drift.   
In addition, a recently obtained formula for the first passage time distribution for skew Brownian motion is given in \cite{AppShel}. 
Considerations of both local time and elastic standard Brownian motion appear in \cite{durrett_restrepo} to cite another biological context for their significance. 
 
The identification of the stochastic particle motions in the presence
of an interface for simulation can have utility in problems involving the computation of particle concentration at a single spatial location, e.g., for so-called resident breakthrough,  since other pde numerical schemes generally involve computation of the entire concentration curve.   Some illustrative
results pertaining to Monte-Carlo simulations of skew diffusions are described in  \cite{Lejay08} and references therein.   In addition, the identification of skew diffusion answers the basic physics question of
finding the particle motion that Jean Perrin would have reported had there been an interface !

From the macro-scale quantity of particle concentrations, 
the determination of $\alpha^*$ in Theorem \ref{thmasympassagetime} may be viewed as
the result of mass conservation.  The interface condition is simply continuity of flux at the 
interface.  However, as will be illustrated by examples in the next section, it is {\it not} necessary that 
continuity of flux always be obeyed in the 
presence of interfaces.  Moreover, even the modeling may 
be at the scale of a
single particle, e.g., biological dispersion of individual animals,  in which `particle concentrations' are not relevant, e.g., see \cite{ovas}, \cite{mc_lewis}, \cite{cos_can} for such examples.  A stochastic particle
model plays an essential role in such situations. 

\section{Related Examples from Other Scientific Fields}

As illustrated by the examples in this section, the role of interfacial phenomena is of much broader
interest than suggested by  advection-dispersion experiments.  However the specific 
nature of the interface can vary, depending on the specific phenomena.    We briefly describe below  three distinct 
classes of examples  of phenomena from the biological/ecological sciences in which such interfaces naturally occur, 
together with compelling questions having substantial biological implications.

\vskip .05in
\noindent{\bf Example B} ({\it\bf Coastal  Upwelling and Fisheries}).  Up-wellings,  the movement of deep nutrient 
rich waters to the sun-lit ocean surface,  occur in roughly
one percent of the ocean but are responsible for nearly fifty-percent of the worlds fishing industry.
The up-welling along the  Malvinas current that occurs off of the coast of Argentina
(see Figure \ref{coastal}) is unusual
in that it is the result of  a very sharp break in the shelf, rather than being driven by winds.  The highlighted points in the figure
represent a flotilla of fishing boats concentrated along the shelf where the up-welling occurs. 
\begin{figure}
\begin{center}
\includegraphics[scale=0.4]{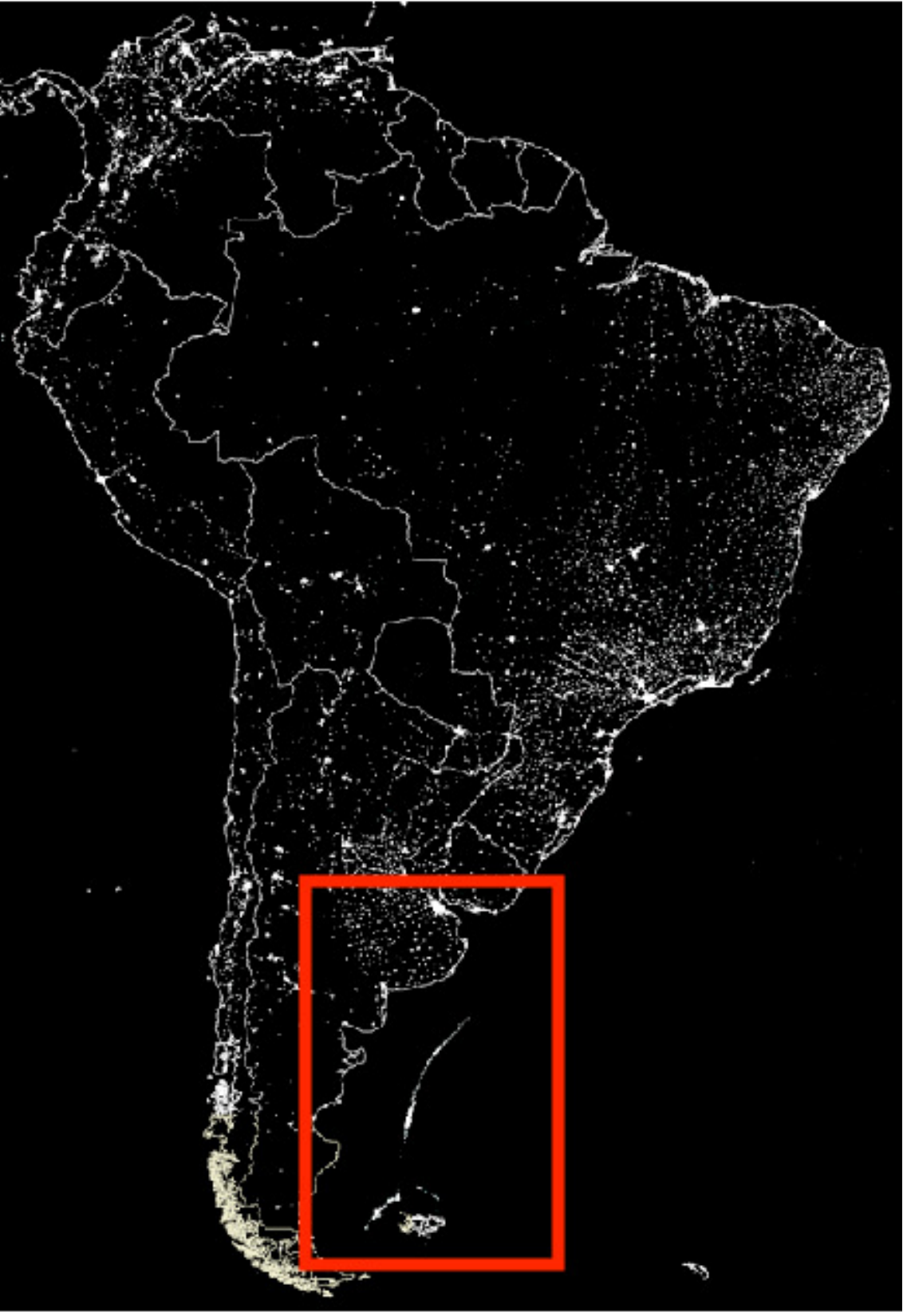}
\end{center}
\caption{\label{coastal} Fishing vessels off the coast of Argentina, 
http://subdude-site.com/WebPics/WebPicsEarthLights/south-america-hilights-429x600.jpg
courtesy  of Ricardo Matano, COAS, Oregon State University}
\end{figure}
 The equation for the free surface elevation $\eta$
as a function of spatial variables $(x,y)$ is derived from principles of geostrophic balance and takes
the  form
\begin{equation}
{\partial\eta\over\partial y} =  -{r\over f}\left({\partial h\over\partial x}\right)^{-1}{\partial^2\eta\over\partial x^2},
\end{equation}
where $r > 0$ and $f < 0$ in the  southern hemisphere, and $h(x)$ is the depth of the ocean at a distance $x$ from 
the shore.  In particular, the sharp break in the shelf makes $h^\prime(x)$ a piecewise constant function
with  positive values $H^+, H^-$. The location of the interface coincides with the  distance to the shelf-break.  In particular, 
if the spatial variable $y > 0$ is viewed as a `time'  parameter, then this is a skew-diffusion equation,
but the physics imply continuity of the derivatives $\partial\eta/\partial x$ at the interface rather than `flux' (see  \cite{Matano} and references therein).

\vskip .05in
\noindent{\bf Example C} ({\it\bf Fender's Blue Butterfly}).  The Fenders Blue is an endangered species of butterfly found
in the pacific northwestern United States.  The primary habitat patch is Kinkaid's Lupin flower
(Figure \ref{butterfly}).

\begin{figure}[htb]
\label{butterfly}
\begin{center}
    \includegraphics[scale=0.7]{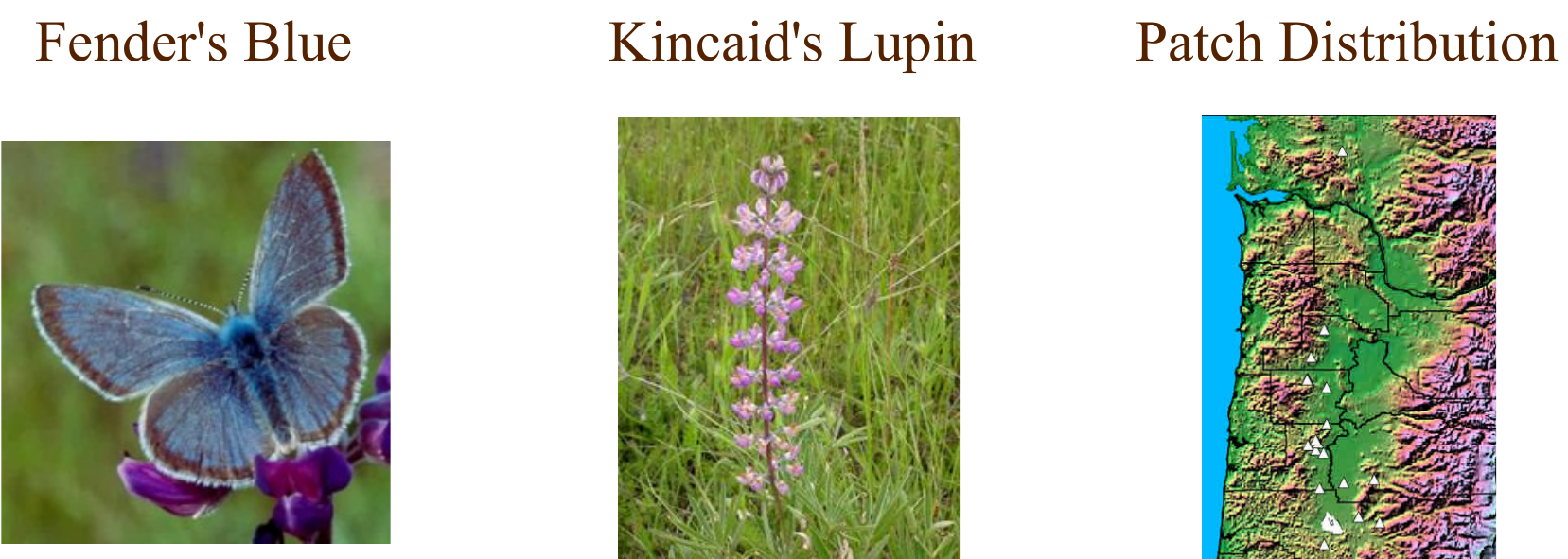}
  \end{center}
\caption{Fender's Blue Butterfly Habitat - Kincaid's Lupin Patches - U.S. Fish and Wildlife Service,Wikipedia http://en.wikipedia.org/wiki/Fender\%27s\_blue\_butterfly, http://www.fws.gov/oregonfwo/Species/PrairieSpecies/gallery.asp}
\end{figure}

   Quoting 
\cite{Schultz},  {\it \lq\lq Given past research on the Fender's blue,  and the potential to investigate response to
patch boundaries, we ask two central questions.  First,  how do organisms respond to habitat edges? Second,
what are the  implications of this behavior  for residence times?\rq\rq}  Sufficiently long residence (occupation) times
in Lupin patches  are required for pollination, eggs, larvae and ultimate sustainability of the population. 
Empirical evidence points  to a skewness in random walk models for butterfly movement at the path boundaries.  The
determination of proper interface conditions is primarily a statistical problem in this application, however we will see
that certain theoretical qualitative analysis may be possible for  setting ranges on interface transmission parameters.
 There is a rapidly growing literature on the  statistical estimation of parameters for 
diffusions, \cite{nakahiro}, \cite{sorensen}, \cite{ait}.   However much (though not all) of this literature is motivated by applications to dispersive models in finance for which the  coefficients are presumed smooth and the data is high frequency.

\vskip .05in
\noindent{\bf Example D} ({\it\bf Sustainability on a River Network}). The movement of larvae in a river system is often modeled by
advective-dispersion equations  in which the rates are determined by hydrologic/geomorphologic relationships in the form
of the  so-called {\it Horton laws}.   In general river networks (Figure \ref{rivernetwork})
 are modeled as directed binary tree graphs and each junction
may be viewed as an interface. There are also special relations known from geomorphology  and hydrology
that can be applied to narrow the class of graphs observed in natural river basins;  e.g., 
see \cite{rinrod}, \cite{peckham}, \cite{gupta}, and references therein.

\begin{figure}[htb]
\label{rivernetwork}
\begin{center}
    \includegraphics[scale=0.7]{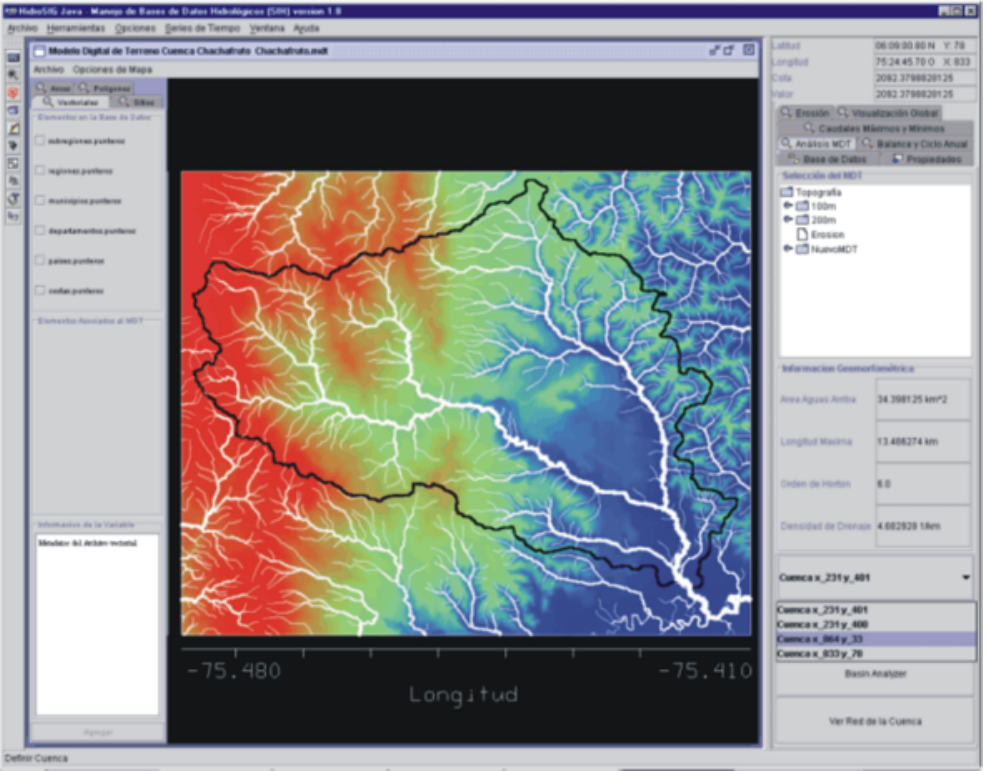}
  \end{center}
\caption{Digitally Generated River Network Using HidroSIG Java;  courtesy of Jorge Ramirez }
\end{figure}
  
Conservation of mass leads to continuity of  flux of larvae across each stream junction as
the  appropriate interface condition.   Problems on sustainability in this context are generally formulated in terms of 
network size and characteristics relative to the production of larvae sufficient to prevent permanent downstream removal
at low population sizes; see 
 \cite{Ramirez11b}  for recent results in the case of a river network, and
 \cite{okada} and \cite{friedlin} for related mathematical considerations.


\section{Natural Dispersion and Natural Occupation Time}

The following theorem provides a useful summary in one-dimension
of the interplay between diffusion coefficients 
 and broader classes of possible 
interfacial conditions.    To set the stage we begin with Brownian motion $X = \{X(t) \equiv \sqrt{D}B(t) + vt:t\ge 0\}$  with constant diffusion coefficient $D > 0$
and constant drift $v$;  
i.e., $B = \{B(t):t\ge 0\}$ is standard Brownian motion starting at zero, 
with unit diffusion coefficient and zero drift.   The stochastic process
$X$ has a number of properties characteristic of diffusion.  Namely, the random path  $t\to X(t)$ is continuous with 
probability one, and 
\begin{equation}
\label{semimart}
\mathbb{E}(X(t) |X(u), u\le s) = X(s) + v(t-s), \qquad 0\le s < t.
\end{equation}
This makes $X$ a continuous semimartingale.  In the case $v=0$, $X$ is in fact said to be a (continuous) martingale.   In general a continuous
semimartingale is a process with continuous paths that differs from 
a (local) martingale by a (unique)
 continuous (adapted) process having finite 
total variation (in the sense of functions of bounded variation).  Such
structure is a natural consequence of the interfaces discussed in this
article. This will be made clear at equation (\ref{skewsde}) below. 

The Brownian motion $X$ is also a Markov process with (homogeneous) transition probabilities 
\begin{equation}
p(t,x,y) = {1\over \sqrt{2\pi Dt}}e^{-{(y-x-vt)^2\over 2Dt}}, \qquad t\ge 0,
\end{equation}
satisfying the advection-dispersion equation 
\begin{equation}
{\partial p\over\partial t} = {\partial\over\partial y}({D\over 2}{\partial p\over\partial y}) - {\partial vp\over\partial y}.
\end{equation}
with continuous derivatives of all orders everywhere.
Since the
displacements $X(t)-X(s)$ are independent for disjoint intervals $(s,t)$,  it follows that the variance of $X(t)$ is linear in  $t$  
and, in particular,  $\mathbb{E}(X(t) - vt)^2 = Dt, t\ge 0.$   The definition of {\it mathematical local time at $a$}, 
denoted $\ell_t^X(a)$, for $X$ can be expressed
as $d\ell_t^X(a)$ is the  amount of time $X$ spends in the infinitesimal neighborhood
 $(a,a \pm da)$ prior to time $t$.  More precisely,
\begin{equation}
\ell^X_t(a) = \lim_{\epsilon\downarrow 0}{D\over 2\epsilon}\int_0^t1_{(a-\epsilon, a+\epsilon)}(X(s))ds.
\end{equation}
Note the presence of the diffusion coefficient $D$. Thus, in the context of dispersion of particles
in a fluid, for example, $[D] = L^2/T$ and $ [\epsilon] =[X]  = L$ yields units $[\ell_t^X] = {1\over L}\times{L^2\over T}\times T = L.$

The standard extension of the mathematical definition of local time for a continuous
semimartingale $X$
 exploits the 
{\it quadratic variation} $<X>_t, t\ge 0,$ of the process.  If $X$ is
a square-integrable martingale then $<X>_t$ is 
 defined by the property that the process
$X^2(t) - <X>_t, t\ge 0,$ is a martingale; see \cite{revuz_yor}.  So, in the case of the Brownian motion one may check that
$<X>_t = Dt, t\ge 0.$  In the case of a continuous semimartingale, the
quadratic variation is quadratic variation of its martingale component.

 Mathematical local time at $a$ can be defined as an increasing continuous
stochastic process $\ell_{a}^X(t), t\ge 0$ such that
\begin{equation}
|X(t) - a| = |X(0) -a| +\int_0^t \sgn(X(s)-a)dX(s) + \ell_t^X(a), \qquad t\ge 0,
\end{equation}
where $\sgn(x) = x/|x|, x\neq 0,$ and $\sgn(0) = -1$ (by convention).  
For purposes of calculation it is often convenient to consider 
(right and left) one-sided  versions defined by
\begin{equation}
\ell_t^{X,+}(a) = \lim_{\epsilon\downarrow 0}{1\over\epsilon}\int_0^t1_{[a,a+\epsilon)}(X(s))d<X>_s,\qquad t\ge 0,
\end{equation}
\begin{equation}
\ell_t^{X,-}(a) = \lim_{\epsilon\downarrow 0}{1\over\epsilon}\int_0^t1_{(a-\epsilon,a]}(X(s))d<X>_s, \qquad t\ge 0.
\end{equation}
Then
\begin{equation}
\ell_t^{X}(a) = {\ell_t^{X,+}(a) + \ell_t^{X,-}(a)\over 2}, \qquad t\ge 0.
\end{equation}
The utility of these quantities rests in the following two formulae:

\noindent{\bf It\^o-Tanaka Formula:}  If $X$ is a continuous 
semimartingale  with local time $\ell_t^X(a), a\in\mathbb{R}, t\ge 0$,
and $f$ is the difference of two convex functions
then
\begin{equation}
f(X(t)) = f(X(0)) + \int_0^t f_-^\prime(X(s))dX(s) + {1\over 2}\int_{\mathbb R}\ell_t^{X,+}(a)f^{\prime\prime}(da),
\end{equation}
where $f^{\prime\prime}(da)$ is a positive measure corresponding
to the second derivative of $f$ in the sense of distributions.
\vskip .15in
\noindent{\bf Occupation Time Formula:} For a non-negative
 Borel function $g$ and $t\ge 0$, one has a.e. that
 \begin{equation}
  \int_0^t g(X(s))d<X>_s = \int_{\mathbb R} g(a)d\ell_t^X(a).
\end{equation}

\vskip .15in
\begin{theorem}
\label{thmalphamart}
 Let $D^+, D^-$ be arbitrary positive numbers and let  $0 <\alpha,  \lambda < 1$.    
Define $Y^{(\alpha)}(t) = \sigma(B_t^{(\alpha)}), t\ge 0$, where $B^{(\alpha)}$ is skew Brownian motion with transmission
parameter $\alpha$ and $\sigma(x) = \sqrt{D^+}x1_{[0,\infty)}(x) + \sqrt{D^-}x1_{(-\infty,0]}(x), \ x\in {\mathbb R}$.
Then 
$$M(t) = f(Y^{(\alpha)}(t)) - {1\over 2}\int_0^t D(Y^{(\alpha)}(u))f^{\prime\prime}(Y(u))du, \quad t\ge 0,$$
is a martingale for all $f\in {\cal D}_\lambda = \{f\in C^2({\mathbb R}\backslash \{0\})\cap C({\mathbb R}) :  \lambda f^\prime(0^+) 
= (1-\lambda)f^\prime(0^-)\}$ if and only if 
$$\alpha = \alpha(\lambda)  = { \lambda\sqrt{D^-}\over \lambda\sqrt{D^-} + (1-\lambda)\sqrt{D^+}}.$$
\end{theorem}
\vskip .15in
\begin{proof}
In the case of skew Brownian motion one has the following relationships between 
mathematical local time at zero and its one-sided variants, e.g., see \cite{Ouknine}
\begin{equation}
\label{skwbmlocrel}
\ell_t^{B^{(\alpha),+}}(0) = 2\alpha\ell_t^{B^{(\alpha)}}(0), 
\qquad \ell_t^{B^{(\alpha),-}}(0) = 2(1-\alpha)\ell_t^{B^{(\alpha)}}(0), \quad t\ge 0.
\end{equation}
Moreover, $B^{(\alpha)}$ is the unique strong solution to the stochastic differential
equation, see \cite{LeGall},
\begin{equation}
\label{skewsde}
dB^{(\alpha)}(t) = {2\alpha-1\over 2\alpha}d\ell_t^{B^{(\alpha)},+}(0) +
dB(t).
\end{equation}
In particular, considering the integrated version, 
$B$ is the martingale component and 
${2\alpha-1\over 2\alpha}\ell_t^{B^{(\alpha)}}$ is the finite variation
component in the view of $B^{(\alpha)}$ as a semimartingale.

It is also straightforward to relate mathematical local times of $Y^{(\alpha)}$ and $B^{(\alpha)}$
through the one-sided formulae as:
\begin{equation*}
\ell_t^{Y^{(\alpha)},+}(0) = \sqrt{D^+}\ell_t^{B^{(\alpha)},+}(0),\qquad 
\ell_t^{Y^{(\alpha)},-}(0) = \sqrt{D^-}\ell_t^{B^{(\alpha)},-}(0)
\end{equation*}
Applying the It\^o-Tanaka formula to the positive and negative parts of $Y^{(\alpha)}$ together with (\ref{skewsde}), one has
\begin{equation*}
d(Y^{(\alpha)}(t))^+ = \sqrt{D^+}\{1_{[B^{(\alpha)}(t)> 0]}dB(t) 
+ {1\over 2}d\ell_t^{B^{(\alpha)},+}(0)\},
\end{equation*}
and
\begin{equation*}
d(Y^{(\alpha)}(t))^- = \sqrt{D^-}\{-1_{[B^{(\alpha)}(t)\le 0]}dB(t) 
- {(2\alpha-1)\over 2\alpha}d\ell_t^{B^{(\alpha)},+}(0)+ {1\over 2}d\ell_t^{B^{(\alpha)},+}(0)\}.
\end{equation*}
Thus, since $Y^{(\alpha)}$ is the difference of its positive and negative parts and noting the cancellation,  
one has
\begin{equation}
\label{natdiffsde}
dY^{(\alpha)}(t) = \sqrt{D(B^{(\alpha)}(t))}dB(t) + 
\left(
\frac{\sqrt{D^+}-\sqrt{D^-}}{2}
+\sqrt{D^-}\frac{2\alpha-1}{2\alpha}\right)
d\ell_t^{B^{(\alpha)},+}(0).
\end{equation}
For a difference of convex functions $f\in{\cal D}_\lambda$,
\begin{equation*}
f^{\prime\prime}(da) = f^{\prime\prime}(a)da + (f^\prime(0^+)-f^\prime(0^-)\delta_0(da).
\end{equation*}
With these preliminaries, again use the It\^o-Tanaka theorem
together with (\ref{skwbmlocrel}) and (\ref{natdiffsde}), to get
\begin{eqnarray}
\nonumber
f(Y^{(\alpha)}(t)) &=& f(Y^{(\alpha)}(0)) + \int_0^tf_-^\prime(Y^{(\alpha)}(s))dY^{(\alpha)}(s)
+ {1\over 2}\int_\mathbb{R}\ell_t^{Y^{(\alpha)},+}(a)f^{\prime\prime}(da)\\
\nonumber
&=& f(Y^{(\alpha)}(0)) + \int_0^tf_-^\prime(Y^{(\alpha)}(s))\sqrt{D(Y^{(\alpha)}(s))}dB(s)\\
\nonumber
&+& {\alpha\sqrt{D^+} -(1-\alpha)\sqrt{D^-}\over2\alpha}\int_0^tf_-^\prime(Y^{(\alpha)}(s))d\ell_s^{B^{(\alpha)},+}(0)\\
\nonumber
&+& {1\over 2}\int_\mathbb{R}\ell_t^{Y^{(\alpha)},+}(a)f^{\prime\prime}(a)da\\
\label{marteqn}
&+&{1\over 2}\int_\mathbb{R}\ell_t^{Y^{(\alpha)},+}(a)(f^\prime(0^+)-f^\prime(0^-)\delta_0(da).
\end{eqnarray} 
According to the mathematical occupation time formula relating
mathematical occupation time and mathematical local time,
 and and noting
that the quadratic variation is given by
$d<Y^{(\alpha)}>_s = D(Y^{(\alpha)}(s))ds$, it also follows that
\begin{equation*}
{1\over 2}\int_\mathbb{R}f^{\prime\prime}(a)d\ell_t^{Y^{(\alpha)},+}(a) = {1\over 2}\int_0^tD(Y^{(\alpha)}(s))f^{\prime\prime}(Y^{(\alpha)}(s))ds.
\end{equation*}
The asserted 
result now follows  by subtracting this term from each side of the
equation (\ref{marteqn}) and noting  the cancellation of local time terms leaving a
stochastic integral with respect to Brownian motion when and only when $\alpha = \alpha(\lambda)$, i.e., 
\begin{eqnarray*}
&\ &{\alpha\sqrt{D^+} -(1-\alpha)\sqrt{D^-}\over2\alpha}\int_0^tf_-^\prime(Y^{(\alpha)}(s))d\ell_s^{B^{(\alpha)},+}(0) + 
{1\over 2}\int_\mathbb{R}\ell_t^{Y^{(\alpha),+}}(a)
(f^\prime(0^+)-f^\prime(0^-))\delta_0(da)\\
&=&  {\alpha\sqrt{D^+} -(1-\alpha)\sqrt{D^-}\over2\alpha}f^\prime(0^-)\ell_t^{B^{(\alpha)},+}(0)
 + 
{1\over 2}\sqrt{D^+}\ell_t^{B^{(\alpha)},+}(0)
(f^\prime(0^+)-f^\prime(0^-))\\
&=& [{\sqrt{D^+}\over 2}f^\prime(0^+) - {1-\alpha\over2\alpha}\sqrt{D^-}f^\prime(0^-)]\ell_t^{B^{(\alpha)},+}(0) = 0
\end{eqnarray*}
if and only if ${1-\alpha\over \alpha}{\sqrt{D^-}\over\sqrt{D^+}}
= {1-\lambda\over\lambda}$.
That this  is enough follows from standard theory, e.g., Theorem 2.4 in \cite{revuz_yor}.
\end{proof}

Some implications for the examples will be described below in terms of the
stochastic particle evolution, however one may also note that one has the following 
consequence at the scale of Kolmogorov's backward equation. 

\begin{corollary}  
Let $D^+, D^-$ be arbitrary positive numbers and let  $0 <\alpha,  \lambda < 1$. Then
for   $c_0\in {\cal D}_\lambda$,  the unique solution to 
$${\partial c\over\partial t} = {1\over 2}D(x){\partial^2 c\over\partial x^2},\quad
\lim_{t\to0^+}c(t,x) = c_0(x), \quad \lambda{\partial c\over\partial x}(t,0^+)
= (1-\lambda){\partial c\over\partial x}(t,0^-), t > 0,$$
is given by
$$c(t,x) = \mathbb{E}_xc_0(Y^{\alpha(\lambda)}(t)), \quad t\ge 0.$$
\end{corollary} 

It is illuminating to consider Theorem \ref{thmalphamart} in the context of the
examples. 

\noindent{\bf Example A}.  Theorem \ref{thmalphamart} provides
 a generalization of the results obtained in 
\cite{Ramirez06} and \cite{App09a}
for the case of advection-dispersion problems across an interface described in Example A. One
may check that  
 \begin{equation}
\lambda = {D^+\over D^+ + D^-}, \qquad
\alpha^* = {\sqrt{D^+}\over\sqrt{D^+} + \sqrt{D^-}}.
\end{equation}
 follow from Theorem \ref{thmalphamart}
 for this application.   This coincides with the results of \cite{Ramirez06} and \cite{App09a}
obtained by other methods.

The following definition is made with reference to both the diffusion coefficient and the interface parameter
in the context of this and the  other examples.

\begin{definition} 
\label{natdiff}
With the choice of $\alpha\equiv\alpha(\lambda)$ given by 
Theorem \ref{thmalphamart}, we
 refer to the process $Y^{(\alpha(\lambda))}$ as the {\it natural diffusion} corresponding to the 
dispersion coefficients $D^+, D^-$ and interface parameter $\lambda$. 
\end{definition}

\noindent{\it\bf Additional Nomenclature:}  We sometimes refer to the natural diffusion corresponding
to $ \lambda = {D^+\over D^+ + D^-}, 
\alpha^* = {\sqrt{D^+}\over\sqrt{D^+} + \sqrt{D^-}}$ as the {\it physical diffusion}.   We refer to the 
diffusion in the case
$\lambda = 1/2, \alpha^\sharp = {\sqrt{D^-}\over\sqrt{D^+} + \sqrt{D^-}}=1-\alpha^*$ as the {\it Stroock-Varadhan diffusion}
since it is the solution to the particular martingale problem originating with these authors; see \cite{SV}.

\vskip .15in

\noindent{\bf Example B}.
Observe that in the application to the coastal up-welling problem one obtains
\begin{equation}
\lambda = 1/2, \qquad \alpha(1/2) = \alpha^\sharp = {\sqrt{D^-}\over\sqrt{D^+} + \sqrt{D^-}}.
\end{equation}
The interface parameter provides continuity of the derivative, rather than of the flux, at the interface.
   The natural diffusion for this
example may be checked to
coincide with the Stroock-Varadhan diffusion in this case.
  Note that the answer to the first passage time problem will be exactly opposite
to that obtained for the advection-dispersion experiments of Example A under this model.

The following modification of the usual notion of mathematical occupation time,
 where the integration
is with respect to quadratic variation  and in units of squared-length,  provides a 
quantity in units of time that we refer to as natural occupation time. \footnote{While we emphasize \lq\lq natural\rq\rq choices from the point
of view of modeling (and units), there are very sound and important reasons for the standard mathematical
definitions. 
 In particular,   no suggestion to change the mathematical definition is intended.  Indeed, as the proof of Theorem \ref{thmalphamart} demonstrates,  the notion of mathematical local time and occupation time and their relationship is extremely powerful in singling out the special value of $\alpha(\lambda)$ for given interface parameter $\lambda$
 and dispersion coefficients $D^\pm$.}

\begin{definition}
Let $X$ be a continuous semimartingale.    The {\it natural occupation time of $G$ by
time $t$},
is defined by  $\tilde\Gamma(G,t) = \int_0^t1(X(s)\in G)ds, t\ge 0,$ for an arbitrary Borel subset $G$ of $\mathbb{R}$.
\end{definition}

The following  result  illustrates another way in which 
the issue raised in Example B relating interfacial conditions
to residence times is indeed a sensitive problem.  The proof
exploits the property of skew Brownian motion
that for any $t>0$,
\begin{equation}
\label{posprob}
P(B^{(\alpha)} > 0) = \alpha.
\end{equation}
  This is easily checked from 
definition and, intuitively, reflects the property that the 
excursion interval $J_{n(t)}$ of $|B|$
containing $t$ must result in a  $[A_{n(t)} = +1]$ coin flip, an
event with probability
$\alpha$.

\begin{theorem}
\label{thmoccuptime}
Let $Y^{(\alpha(\lambda))}$  denote the natural diffusion for the dispersion 
coefficients $D^+, D^-$ and interface parameter $\lambda$.   Denote natural occupation time
processes by
$$\tilde{\Gamma}_\lambda^+(t) = \int_0^t 1[Y^{(\alpha(\lambda))}(s) > 0]ds, \quad t\ge 0.$$
Similarly let $\tilde{\Gamma}_\lambda^-(t) = t- \tilde{\Gamma}_\lambda^+(t), t\ge 0.$  Then,
$$\mathbb{E}\tilde{\Gamma}_\lambda^+(t) > \mathbb{E}\tilde{\Gamma}_\lambda^-(t) \ \forall t>0\quad {\iff} \lambda > {\sqrt{D^+}\over\sqrt{D^+} +\sqrt{D^-}},$$
with equality  when $\lambda = {\sqrt{D^+}\over\sqrt{D^+} +\sqrt{D^-}}$.
\end{theorem}
\begin{proof}
Using the definition of natural diffusion $Y^{(\alpha(\lambda))}$ for the
parameters $D^\pm, \lambda$ and the above noted property
(\ref{posprob}) of skew Brownian motion, one has
\begin{eqnarray*}
\mathbb{E}\tilde\Gamma_\lambda^+(t) &:=&\mathbb{E} \int_0^t1[Y^{(\alpha(\lambda))}(s)>0]ds\\
&=&\mathbb{E} \int_0^t1[\sqrt{D^+}B^{(\alpha(\lambda))}(s))>0]ds\\
&=& \int_0^tP(B^{(\alpha(\lambda))}(s)> 0)ds\\
&=&  t\alpha(\lambda).
\end{eqnarray*}
Similarly $\mathbb{E}\tilde\Gamma_\lambda^-(t) = 
t(1-\alpha(\lambda))$.  
Thus
\begin{equation*}
{\mathbb{E}\tilde\Gamma_\lambda^+(t)\over\mathbb{E}\tilde\Gamma_\lambda^-(t)}
= ({\alpha(\lambda)\over1-\alpha(\lambda)})= ({\lambda\sqrt{D^-}\over
(1-\lambda)\sqrt{D^+}}).
\end{equation*}
The assertion now follows.
\end{proof}

\vskip .2in

\noindent{\bf Example A vs Example C}.
As noted previously, in the Example C pertaining to insect movement,  the determination of an interface 
condition is largely a statistical issue as there is no scientific rationale
 to apply mass conservation principles, or smooth Fickean flux laws.  
  In fact, it is
 interesting to observe that under the mass conservation one arrives at the
 interface parameter 
 \begin{equation}
 \lambda = D^+/(D^+ + D^-),
 \end{equation}
for which  the residence time is longer on average in the region
 with the faster dispersion rate!  While this is to be anticipated 
 for physical experiments of dispersion
 in porous media of the type described in Example A,  
 Theorem \ref{thmoccuptime} shows that in fact the 
 conservative interface condition (defined by this choice of $\lambda$)  
 would not be appropriate for 
 models of animal movement for which the faster dispersion occurs in
more hostile environments; e.g., 
see \cite{mc_lewis}, \cite{cos_can}, \cite{ovas}  for relevant considerations.  There is indeed something to be learned from data as to what exactly might apply, but such theoretical insights can provide a useful guide, and help to prevent mistaken assumptions when transferring more well-developed physical principles to biological/ecological phenomena.

\vskip .15in

\section{Continuity of Natural Local Time} 

We now discuss another issue pertaining to the definition of
mathematical local time and
the basis for the suggested modification to natural local time.

In the often cited article \cite{walsh}, it was first observed that skew-Brownian motion provides a non-trivial example of a continuous
semimartingale on the interval $(-\infty,\infty)$
having a discontinuous mathematical local time.
Let us now examine this situation in the context
of natural diffusions.  For the purposes of this
discussion take the drift $v = 0$ and consider a
constant diffusion coefficient $D^+ = D^- = D$

Since the quadratic variation of skew Brownian motion coincides with
that of Brownian motion, one has
$<\sqrt{D}B^{(\alpha)}> \equiv <\sqrt{D}B>.$
Among this class of natural diffusions,
one may ask what distinguishes the  particular diffusion 
$X \equiv \sqrt{D}B$ ?
Of course, the answer is that $X$ is determined by
$\alpha = \lambda = 1/2$.
If one views this choice in the context of the flux in
particle concentration, then it provides continuity of flux.
On the other hand, in view of the respective theorems of
Trotter \cite{trotter} and Walsh \cite{walsh},
it is also the unique choice
of $\alpha$ from among all skew Brownian motions
to make local time continuous.
The latter may be viewed as a stochastic particle determination
of the physical diffusion model, among natural diffusions, for constant diffusion coefficient
$D$.

The next theorem,  a version of which was originally conceived in \cite{App09b},
extends this to the more general framework of the present paper,
in particular to include the case $D^+\neq D^-$.
However, it requires the
following modification of the definition of mathematical local time,
referred to here as natural local time.

\begin{definition}
\label{natloctime}
Let $X$ be a continuous semimartingale.  The {\it natural local time at $a$} $\tilde\ell_t^X = {\tilde\ell_t^{X,+} + \tilde\ell_t^{X,-}\over 2}$
 of $X$ is defined by
\begin{equation}
\tilde\ell_t^{X,+}(a) = \lim_{\epsilon\downarrow 0}{1\over \epsilon}\int_0^t1_{[a, a+\epsilon)}(X(s))ds, \quad
\tilde\ell_t^{X,-}(a) = \lim_{\epsilon\downarrow 0}{1\over \epsilon}\int_0^t1_{(a-\epsilon, a]}(X(s))ds
\end{equation}
providing the indicated limits exist almost surely.
\end{definition}

The units of natural local time are then $[\tilde\ell_t^X] = {T\over L}$, appropriate to  a measurement of (occupation) time in the vicinity
of a spatial location $a$.
While the purpose here is not to explore the generality for which natural local time exists among
all continuous semimartingales,  according to the following theorem
 it does exist for natural skew diffusion.  Moreover, continuity has a special significance.

\begin{theorem}
\label{natloctimeofnatdiff}
Let $Y^{(\alpha(\lambda))}$ be the natural skew diffusion with parameters $D^{\pm}, \lambda$.  
Then the natural
local time of $Y^{(\alpha(\lambda))}$ at $0$ 
is continuous if and only if 
 $\lambda = {D^+\over D^+ +D^-}$, i.e., if and only if $\alpha(\lambda)=\alpha^*$
 and thus
 $Y^{(\alpha^*)}$ is the physical diffusion.
 \end{theorem}
\begin{proof}
In view of (\ref{skwbmlocrel}) observe that
\begin{equation*}
{\ell_t^{B^{(\alpha)},+}(0)\over \ell_t^{B^{(\alpha)},-}(0)}
= {\alpha\over 1-\alpha}, \qquad \forall t\ge 0.
\end{equation*}
Moreover,
\begin{eqnarray*}
\tilde\ell_t^{B^{(\alpha)},+}(0) &=& 
\lim_{\epsilon\downarrow 0}{1\over\epsilon}\int_0^t 1[0\le Y^{(\alpha)}(s) < \epsilon]ds\\
&=& {1\over\sqrt{D^+}}\lim_{\epsilon\downarrow 0}{\sqrt{D^+}\over\epsilon}\int_0^t 1[0\le B^{(\alpha)}(s)< {\epsilon\over\sqrt{D^+}}]ds\\
&=& {1\over\sqrt{D^+}}\ell_t^{B^{(\alpha)},+}(0).
\end{eqnarray*}
Similarly, $\tilde\ell_t^{B^{(\alpha)},-}(0) = {1\over\sqrt{D^-}}\ell_t^{B^{(\alpha)},-}(0).$  Thus,
\begin{equation*}
{\tilde\ell_t^{B^{\alpha(\lambda)},+}(0)\over \tilde\ell_t^{B^{\alpha(\lambda)},-}(0)}
= {\alpha(\lambda)\over 1-\alpha(\lambda)}{\sqrt{D^-}\over\sqrt{D^+}}.
\end{equation*}
Now simply observe that this ratio is one if and only if
$\lambda D^- = (1-\lambda)D^+$, which establishes the assertion.
\end{proof}

\section{Acknowledgments.} 
 The authors are grateful to Jorge Ramirez for several technical comments that improved the
 content and exposition.  The first author was partially supported by
an NSF-IGERT-0333257 graduate training grant in ecosystems informatics at Oregon State 
University, and the remaining authors were partially supported by a grant DMS-1122699 from 
the National Science Foundation.  The corresponding author is also grateful for support provided 
by the Courant Institute of Mathematical Sciences,  New York University, during preparation of the 
final draft of this  article. 

\end{document}